\documentclass{amsart}
\usepackage{amsmath}
\usepackage{amsfonts}

\newtheorem{theorem}{Theorem}
\theoremstyle{plain}

\newtheorem{corollary}{Corollary}

\newtheorem{definition}{Definition}
\newtheorem{example}{Example}

\newtheorem{lemma}{Lemma}

\newtheorem{remark}{Remark}

\numberwithin{equation}{section}
\input{tcilatex}

\begin{document}
\title[Jensen, Jensen-Steffensen and various types of convexity]{Improvement
of Jensen, Jensen-Steffensen's and Jensen's functionals related inequalities
for various types of convexity.}
\author{Shoshana Abramovich}
\address{Department of Mathematics, University of Haifa, Haifa, Israel}
\email{abramos@math.haifa.ac.il}
\date{January 1, 2025 }
\subjclass{26D15,}
\keywords{Convexity, Uniform convexity, $\Phi $-convexity, strong convexity,
superquadracity, Jensen's differences, Jensen inequality,
Jensen-Steffensen's inequality.}

\begin{abstract}
In this paper we deal with improvement of Jensen, Jensen-Steffensen's and
Jensen's functionals related inequalities for uniformly convex, $\Phi $%
-convex and superquadratic functions.
\end{abstract}

\maketitle

\section{\textbf{\ Introduction}}

In this paper we deal with improvement of Jensen, Jensen-Steffensen's and
Jensen's differences related inequalities for uniformly convex, $\Phi $%
-convex and superquadratic functions.

.

We adopt the technique and ideas in \cite{ASIP}, \cite{BL} and \cite{AD},
with necessary changes, to get much more extended results.

In Section 2 we extend the results in \cite{ASIP}, there on superquadracity,
to include also uniform convexity.

In Section 3 we extend the results in \cite{BL}, there on strongly convex
functions, to include also uniform convexity, $\Phi $-convexity and
superquadratic functions.

In Section 4 we improve the result of \cite{AD} related to the difference
between two Jensen's functionals

The first type of convexity we deal with is \textbf{uniform convexity}:

\begin{definition}
\label{Def1} \cite{A} Let $\left[ a,b\right] \subset 
\mathbb{R}
$ be an interval and $\Phi :\left[ 0,b-a\right] \rightarrow 
\mathbb{R}
$ be a function. A function $f:\left[ a,b\right] \rightarrow 
\mathbb{R}
$ is said\ to be generalized $\Phi $-uniformly convex\ if: 
\begin{eqnarray*}
tf\left( x\right) +\left( 1-t\right) f\left( y\right) &\geq &f\left(
tx+\left( 1-t\right) y\right) +t\left( 1-t\right) \Phi \left( \left\vert
x-y\right\vert \right) \\
\text{for \ }x,y &\in &\left[ a,b\right] \text{ \ and }t\in \left[ 0,1\right]
\text{.}
\end{eqnarray*}%
If in addition $\Phi \geq 0$, then $f$\ is said to be $\Phi $\textbf{%
-uniformly convex}\textit{, }or\textit{\ }\textbf{uniformly convex with
modulus }$\Phi $. In the special case where $\Phi \left( \left\vert
x-y\right\vert \right) =c\left( x-y\right) ^{2}$, $c>0$, $f$ is called 
\textbf{strongly convex} function.
\end{definition}

\begin{remark}
\label{Rem1} It is proved in \cite{Z} and in \cite{N} that when $f$ is
uniformly convex,\textbf{\ }there is always a modulus $\Phi $ which is
increasing and $\Phi \left( 0\right) =0$. In this paper we use such a $\Phi $
when dealing with uniformly convex function.

It is also shown in \cite{Z} that the inequality%
\begin{equation}
f\left( \sum_{r=1}^{n}\lambda _{r}x_{r}\right) \leq \sum_{r=1}^{n}\lambda
_{r}\left( f\left( x_{r}\right) -\Phi \left( \left\vert
x_{r}-\sum_{j=1}^{n}\lambda _{j}x_{j}\right\vert \right) \right)  \label{1.1}
\end{equation}%
holds, when $0\leq \lambda _{i}\leq 1,$ $\sum_{i=1}^{n}\lambda _{i}=1$.
\end{remark}

\begin{remark}
\label{Rem2} The functions $f\left( x\right) =x^{n},$ $x\geq 0,$ $n=2,3...$,
are uniformly convex on $x\geq 0$ with a modulus $\Phi \left( x\right)
=x^{n},$ $x\geq 0$ (see \cite{A}), (these functions are also superquadratic).
\end{remark}

\begin{remark}
\label{Rem3} In \cite[Theorem 2.1]{Z}, and \cite[Theorem 1, Inequality (23)]%
{N}, it is proved that the set of uniformly convex functions $f$\ defined on 
$\left[ a,b\right] $, with modulus $\Phi $,\ which are continuously
differentiable, satisfy\ the inequality 
\begin{equation}
f\left( y\right) -f\left( x\right) \geq f^{^{\prime }}\left( x\right) \left(
y-x\right) +\Phi \left( \left\vert y-x\right\vert \right) ,\quad \Phi :\left[
0,\left( b-a\right) \right] \rightarrow 
\mathbb{R}
_{+},  \label{1.2}
\end{equation}%
for $x,y\in \left[ a,b\right] $.
\end{remark}

The authors of \cite{BL} (here Theorem \ref{Th1} below), deal only with
strongly convex functions. In Section 3 we use similar technique and ideas
to get analog theorems for three types of convexity.

\begin{theorem}
\label{Th1} \cite{BL} Let $f:\left( a,b\right) \rightarrow 
\mathbb{R}
$ be a strongly convex function with modulus $c>0$. Suppose $\mathbf{x}%
=\left( x_{1},...,x_{n}\right) \in \left( a,b\right) ^{n\text{ }}$ and $%
\mathbf{a}=\left( a_{1},...,a_{n}\right) $ is a nonnegative $n$-tuple with $%
A_{n}=\sum_{i=1}^{n}a_{i}>0$. Let $\overline{x}=\frac{1}{A_{n}}%
\sum_{i=1}^{n}a_{i}x_{i}$ and $\widehat{x}_{i}=\left( 1-\lambda _{i}\right) 
\overline{x}+\lambda _{i}x_{i},$ $\lambda _{i}\in \left[ 0,1\right] ,$ $i\in
\left\{ 1,...,n\right\} $. Then%
\begin{eqnarray*}
0 &\leq &\left\vert \frac{1}{A_{i}}\sum_{i=1}^{n}a_{i}\left\vert f\left(
x_{i}\right) -f\left( \widehat{x}_{i}\right) -c\left( 1-\lambda _{i}\right)
^{2}\left( \overline{x}-x_{i}\right) ^{2}\right\vert \right. \\
&&\left. -\frac{1}{A_{n}}\sum_{i=1}^{n}a_{i}\left( 1-\lambda _{i}\right)
\left\vert f^{^{\prime }}\left( \widehat{x}_{i}\right) \right\vert
\left\vert \overline{x}-x_{i}\right\vert \right\vert \\
&\leq &\frac{1}{A_{n}}\sum_{i=1}^{n}a_{i}f\left( x_{i}\right) -\frac{1}{A_{n}%
}\sum_{i=1}^{n}a_{i}f\left( \widehat{x}_{i}\right) \\
&&-\frac{1}{A_{n}}\sum_{i=1}^{n}a_{i}\left( 1-\lambda _{i}\right)
f^{^{\prime }}\left( \widehat{x}_{i}\right) \left( \overline{x}-x_{i}\right)
-c\sum_{i=1}^{n}a_{i}\left( 1-\lambda _{i}\right) ^{2}\left( \overline{x}%
-x_{i}\right) ^{2}.
\end{eqnarray*}%
and 
\begin{eqnarray*}
&&\frac{1}{A_{n}}\sum_{i=1}^{n}a_{i}f\left( x_{i}\right) -\frac{1}{A_{n}}%
\sum_{i=1}^{n}a_{i}f\left( \left( 1-\lambda _{i}\right) \overline{x}+\lambda
_{i}x_{i}\right) \\
&\leq &\frac{1}{A_{n}}\sum_{i=1}^{n}a_{i}\left( 1-\lambda _{i}\right)
f^{^{\prime }}\left( x_{i}\right) \left( x_{i}-\overline{x}\right) -\frac{c}{%
A_{n}}\sum_{i=1}^{n}a_{i}\left( 1-\lambda _{i}\right) ^{2}\left( \overline{x}%
-x_{i}\right) ^{2}
\end{eqnarray*}%
hold.
\end{theorem}

The second type of convexity dealt with in this paper is $\Phi $-convexity:

\begin{definition}
\label{Def2} \cite{GP} A real value function $f$ defined on a real interval $%
\left[ a,b\right] $ is called $\Phi $\textbf{-convex} if for all $x,y\in %
\left[ a,b\right] $, $t\in \left[ 0,1\right] $ it satisfies%
\begin{eqnarray*}
&&tf\left( x\right) +\left( 1-t\right) f\left( y\right) +t\Phi \left( \left(
1-t\right) \left\vert x-y\right\vert \right) +\left( 1-t\right) \Phi \left(
t\left\vert x-y\right\vert \right) \\
&\geq &f\left( tx+\left( 1-t\right) y\right) ,
\end{eqnarray*}%
where $\Phi :\left[ 0,b-a\right] \rightarrow 
\mathbb{R}
_{+},$\ is called \textbf{error function}.
\end{definition}

In \cite[Theorem 3.1]{GP} it is proved that:

\begin{corollary}
\label{Cor1} Let $f$ be $\Phi $-convex function on $\left[ a,b\right] $
then, there exists a function $\varphi :\left[ a,b\right] \rightarrow 
\mathbb{R}
$ such that for all $x,u\in \left[ a,b\right] $,%
\begin{equation*}
f\left( u\right) +\left( x-u\right) \varphi \left( u\right) \leq f\left(
x\right) +\Phi \left( \left\vert u-x\right\vert \right) .
\end{equation*}%
Also, when $x_{i}\in \left[ a,b\right] ,$ $i=1,...,n,$ $0\leq \lambda
_{i}\leq 1,$ $i=1,...,n,$ $\sum_{i=1}^{n}\lambda _{i}=1$, 
\begin{equation*}
f\left( \sum_{r=1}^{n}\lambda _{r}x_{r}\right) \leq \sum_{r=1}^{n}\lambda
_{r}\left( f\left( x_{r}\right) +\Phi \left( \left\vert
x_{r}-\sum_{j=1}^{n}\lambda _{j}x_{j}\right\vert \right) \right) ,
\end{equation*}%
holds.
\end{corollary}

The third type of convexity is superquadracity:

\begin{definition}
\label{Def3} \cite[Definition 1]{AJS} A function $\varphi :\left[ 0,\infty
\right) \rightarrow 
\mathbb{R}
$ is \textbf{superquadratic} provided that for all $x\geq 0$\ there exists a
constant $C_{x}\in \ 
\mathbb{R}
$ such that%
\begin{equation}
\varphi \left( y\right) \geq \varphi \left( x\right) +C_{x}\left( y-x\right)
+\varphi \left( \left\vert y-x\right\vert \right)  \label{1.3}
\end{equation}%
for all $y\geq 0$.\ If the reverse of (\ref{1.3}) holds then $\varphi $\ is
called subquadratic.
\end{definition}

\begin{remark}
\bigskip\ \label{Rem4} \cite{AJS} If $\varphi $ is superquadratic, then the
inequality%
\begin{equation*}
\varphi \left( \sum_{r=1}^{n}\lambda _{r}x_{r}\right) \leq
\sum_{r=1}^{n}\lambda _{r}\left( \varphi \left( x_{r}\right) -\varphi \left(
\left\vert x_{r}-\sum_{j=1}^{n}\lambda _{j}x_{j}\right\vert \right) \right)
\end{equation*}%
holds for $x_{r}\geq 0,$ $\lambda _{r}\geq 0,$ \ $r=1,...,n$ and $%
\sum_{r=1}^{n}\lambda _{r}=1.$
\end{remark}

Analogous to Lemma \ref{Lem1} and Theorem \ref{Th2} regarding
superquadracity, we prove in Section 2, using similar technique and adding
the definition of $H$-superadditivity, a lemma and a theorem for uniformly
convex functions.

\begin{lemma}
\label{Lem1} \cite[Lemma 1]{ASIP} Let $f$ be continuously differentiable 
\textit{on }$[0,b)$ and $f^{\prime }$ be superadditive \textit{on }$[0,b)$.
Then the function $D:\left[ 0,b\right) \rightarrow 
\mathbb{R}
$ defined by%
\begin{equation*}
D\left( y\right) =f\left( y\right) -f\left( z\right) -f^{\prime }\left(
z\right) \left( y-z\right) -f\left( \left\vert y-z\right\vert \right)
+f\left( 0\right)
\end{equation*}%
\textit{is nonnegative on }$\left[ 0,b\right) $\textit{, nonincreasing on }$%
\left[ 0,z\right) $\textit{, and nondecreasing on }$\left[ z,b\right) $%
\textit{, for }$0\leq z<b$\textit{,} and $D.$\newline
If also $f\left( 0\right) \leq 0,$\ we have%
\begin{equation*}
f\left( y\right) -f\left( z\right) -f\left( \left\vert y-z\right\vert
\right) \geq f^{\prime }\left( z\right) \left( y-z\right) -f\left( 0\right)
\geq f^{\prime }\left( z\right) \left( y-z\right) .
\end{equation*}%
Taking $C(x)=f^{\prime }\left( z\right) $ we see that $f$ is
superquadratic.\bigskip
\end{lemma}

\begin{theorem}
\label{Th2} \cite[Theorem 1]{ASIP} Let $f$ be continuously differentiable on 
$\left[ 0,b\right) $ and $f^{\prime }$ be superadditive on $\left[
0,b\right) .$ Let $\boldsymbol{a}$ be a real $n$-tuple \textit{satisfying} $%
0\leq \sum_{i=1}^{j}a_{i}\leq \sum_{i=1}^{n}a_{i}$, $\sum_{i=1}^{n}a_{i}>0$
and $x_{i}\in \left[ 0,b\right) $, $i=1,...,n$ be such that $x_{1}\leq
x_{2}\leq ...\leq x_{n}.$ Then

\textbf{a.}%
\begin{equation*}
f\left( c\right) -f\left( 0\right) +f^{\prime }\left( c\right) \left( 
\overline{x}-c\right) +\frac{1}{A_{n}}\sum_{i=1}^{n}a_{i}f\left( \left\vert
x_{i}-c\right\vert \right) \leq \frac{1}{A_{n}}\sum_{i=1}^{n}a_{i}f\left(
x_{i}\right)
\end{equation*}%
holds for all $c\in \left[ 0,b\right) $ where $\overline{x}=\frac{1}{A_{n}}%
\sum_{i=1}^{n}a_{i}x_{i}.$

\textbf{b.} If in addition $f\left( 0\right) \leq 0,$ then $f$\ is
superquadratic and%
\begin{equation*}
\frac{1}{A_{n}}\sum_{i=1}^{n}a_{i}f\left( x_{i}\right) \geq f\left( c\right)
+f^{\prime }\left( c\right) \left( \overline{x}-c\right) +\frac{1}{A_{n}}%
\sum_{i=1}^{n}a_{i}f\left( \left\vert x_{i}-c\right\vert \right) .
\end{equation*}%
In particular%
\begin{equation}
\frac{1}{A_{n}}\sum_{i=1}^{n}a_{i}f\left( x_{i}\right) \geq f\left( 
\overline{x}\right) +\frac{1}{A_{n}}\sum_{i=1}^{n}a_{i}f\left( \left\vert
x_{i}-\overline{x}\right\vert \right) .  \label{1.4}
\end{equation}

\textbf{c. }If in addition $f\geq 0$ and $f(0)=f^{^{\prime }}(0)=0,$ then $f$
is convex increasing and superquadratic and 
\begin{equation*}
\frac{1}{A_{n}}\sum_{i=1}^{n}a_{i}f\left( x_{i}\right) -f\left( c\right)
-f^{\prime }\left( c\right) \left( \overline{x}-c\right) \geq \frac{1}{A_{n}}%
\sum_{i=1}^{n}a_{i}f\left( \left\vert x_{i}-c\right\vert \right) \geq 0.
\end{equation*}
\end{theorem}

In Section 2 we adapt the ideas and technique of Theorem \ref{Th2} to get
new results. For this we define the following:

\begin{definition}
\label{Def4} Let the functions $g:\left[ a,b\right] \rightarrow 
\mathbb{R}
$ and $H:\left[ 0,b-a\right] \rightarrow 
\mathbb{R}
$ be continuosly differentiable. We name $g,$ $H$\textbf{-superadditive }if
for all $x,y\in \left[ a,b\right] $ with $x\leq y$, when the inequality 
\begin{equation*}
g\left( y\right) -g\left( x\right) \geq H\left( y-x\right)
\end{equation*}%
holds.
\end{definition}

\bigskip In Sections 2 and 3 our proofs rely on inequalities of the type

\begin{equation*}
f\left( x\right) -f\left( u\right) \geq \varphi \left( u\right) \left(
x-u\right) +\Phi \left( \left\vert x-u\right\vert \right) ,
\end{equation*}%
satisfied according to Remark \ref{Rem3}, Corollary \ref{Cor1} and
Definition \ref{Def3} for uniform convex, $\Phi $-convex and superquadratic
functions respectively.

In Section 4 we prove\ results on the difference between a normalized Jensen
functional and another normalized Jensen functional for superquadratic and
uniformly convex functions.

Jensen functional is:

\begin{equation*}
J_{n}\left( f,\mathbf{x},\mathbf{p}\right) =\sum_{i=1}^{n}p_{i}f\left(
x_{i}\right) -f\left( \sum_{i=1}^{n}p_{i}x_{i}\right) .
\end{equation*}

S. S. Dragomir proved:

\begin{theorem}
\label{Th3} \cite{D} \textit{Consider the normalized Jensen functional where 
}$f:C\longrightarrow 
\mathbb{R}
$\textit{\ is a convex function on the convex set }$C$ in a real linear
space,\textit{\ }$\mathbf{x}=\left( x_{1},...,x_{n}\right) \in C^{n},$ and%
\textit{\ \ }$\mathbf{p}=\left( p_{1},...,p_{n}\right) ,$\textit{\ \ }$%
\mathbf{q}=\left( q_{1},...,q_{n}\right) $\textit{\ are non-negative
n-tuples satisfying }$\sum_{i=1}^{n}p_{i}=1,$\textit{\ \ }$%
\sum_{i=1}^{n}q_{i}=1,$\textit{\ \ }$p_{i},$\textit{\ }$q_{i}>0,$\textit{\ \ 
}$i=1,...,n$\textit{. Then } 
\begin{equation}
MJ_{n}\left( f,\mathbf{x},\mathbf{q}\right) \geq J_{n}\left( f,\mathbf{x},%
\mathbf{p}\right) \geq mJ_{n}\left( f,\mathbf{x},\mathbf{q}\right) ,\qquad
\label{1.5}
\end{equation}%
provided that 
\begin{equation}
m=\min_{1\leq i\leq n}\left( \frac{p_{i}}{q_{i}}\right) ,\quad M=\max_{1\leq
i\leq n}\left( \frac{p_{i}}{q_{i}}\right) .  \label{1.6}
\end{equation}
\end{theorem}

In \cite{AD} the following two theorems are proved:

\begin{theorem}
\label{Th4} Under the same conditions and definitions on\ $\ \mathbf{p},$ $%
\mathbf{q},$ $\mathbf{x},$ $m$ and $M$ as in Theorem \ref{Th3}, if $I$ is $%
\left[ 0,a\right) $ or $\left[ 0,\infty \right) $ and $f$ is a
superquadratic function on $I$, then 
\begin{eqnarray}
&&J_{n}\left( f,\mathbf{x},\mathbf{p}\right) -mJ_{n}\left( f,\mathbf{x},%
\mathbf{q}\right)  \label{1.7} \\
&\geq &mf\left( \left\vert \sum_{i=1}^{n}\left( q_{i}-p_{i}\right)
x_{i}\right\vert \right) +\sum_{i=1}^{n}\left( p_{i}-mq_{i}\right) f\left(
\left\vert x_{i}-\sum_{j=1}^{n}p_{j}x_{j}\right\vert \right)  \notag
\end{eqnarray}%
and 
\begin{eqnarray}
&&J_{n}\left( f,\mathbf{x},\mathbf{p}\right) -MJ_{n}\left( f,\mathbf{x},%
\mathbf{q}\right)  \label{1.8} \\
&\leq &-\sum_{i=1}^{n}\left( Mq_{i}-p_{i}\right) f\left( \left\vert
x_{i}-\sum_{j=1}^{n}q_{j}x_{j}\right\vert \right) -f\left( \left\vert
\sum_{i=1}^{n}\left( p_{i}-q_{i}\right) x_{i}\right\vert \right) .  \notag
\end{eqnarray}
\end{theorem}

In the sequel we use the following notations:

Let $\mathbf{x}_{\uparrow }=\left( x_{\left( 1\right) },...,x_{\left(
n\right) }\right) $ be the \textit{increasing rearrangement} of $\mathbf{x=}%
\left( x_{1},...,x_{n}\right) .$ $\ $Let\ $\pi $ be the permutation that
transfers $\mathbf{x}$ into $\mathbf{x}_{\uparrow }$\ and let $\left( 
\overline{p}_{1},...,\overline{p}_{n}\right) $ and $\left( \overline{q}%
_{1},...,\overline{q}_{n}\right) $ be the $n$-tuples obtained by the same
permutation \ $\pi $\ \ on $\left( p_{1},...,p_{n}\right) $ and $\left(
q_{1},...,q_{n}\right) $ respectively. Then for an $n$-tuple \ $\mathbf{x}%
=\left( x_{1},...,x_{n}\right) ,$\ \ $x_{i}\in I$\ ,\ \ $i=1,...,n$ \ where $%
I$ \ is an interval in $%
\mathbb{R}
$\ we get the following results:

\begin{theorem}
\label{Th5} Let $\mathbf{p}=\left( p_{1},...,p_{n}\right) ,$ where $\ 0\leq
\sum_{j=1}^{i}\overline{p}_{j}\leq 1,$\ \ $i=1,...,n,$\ \ $%
\sum_{i=1}^{n}p_{i}=1,$ and\ $\mathbf{q}=\left( q_{1},...,q_{n}\right) ,$ \ $%
0<\sum_{j=1}^{i}\overline{q}_{j}<1,$\ \ $i=1,...,n-1$,\ \ $%
\sum_{i=1}^{n}q_{i}=1,$ and $\ \mathbf{p}\neq $\ \ $\mathbf{q}.$ \ Denote 
\begin{equation}
m_{i}=\frac{\sum_{j=1}^{i}\overline{p}_{j}}{\sum_{j=1}^{i}\overline{q}_{j}}%
,\qquad \overline{m}_{i}=\frac{\sum_{j=i}^{n}\overline{p}_{j}}{\sum_{j=i}^{n}%
\overline{q}_{j}},\qquad i=1,...,n  \label{1.9}
\end{equation}%
where $\left( \overline{p}_{1},...,\overline{p}_{n}\right) $ and $\left( 
\overline{q}_{1},...,\overline{q}_{n}\right) $ are as denoted above, and 
\begin{equation}
m^{\ast }=\min_{1\leq i\leq n}\left\{ m_{i},\overline{m_{i}}\right\} ,\qquad
M^{\ast }=\max_{1\leq i\leq n}\left\{ m_{i},\overline{m_{i}}\right\} .
\label{1.10}
\end{equation}%
If \ $\mathbf{x}=\left( x_{1},...,x_{n}\right) $\ \ is any n-tuple in $%
I^{n}, $\ \ where $I$\ is an interval in $%
\mathbb{R}
,$ then 
\begin{equation}
M^{\ast }J_{n}\left( f,\mathbf{x},\mathbf{q}\right) \geq J_{n}\left( f,%
\mathbf{x},\mathbf{p}\right) \geq m^{\ast }J_{n}\left( f,\mathbf{x},\mathbf{q%
}\right) ,  \label{1.11}
\end{equation}%
where\ \ $f:I\longrightarrow 
\mathbb{R}
$\ \ is a convex function on the interval $I$.
\end{theorem}

Using Theorem \ref{Th4} and Theorem \ref{Th5} we prove in Theorem \ref{Th19}
that Theorem \ref{Th4} is valid also when instead of (\ref{1.6}) we impose
different conditions given by (\ref{1.9}) and (\ref{1.10}).

\section{\textbf{Improvement of Jensen-Steffensen inequality for uniformly
convex functions}}

Analogous to Lemma \ref{Lem1} and Theorem \ref{Th2} we prove the following
Lemma \ref{Lem2} and Theorem \ref{Th6}. We use a technique similar to that
used in \cite{ASIP} where instead of $f^{^{\prime }}$ being superadditive,
here in Lemma \ref{Lem2} and Theorem \ref{Th6}, $f^{^{\prime }}$ is $\Phi
^{^{\prime }}$-superadditive (see Definition \ref{Def4}).\bigskip

\begin{lemma}
\label{Lem2} Let $f$ be continuously differentiable \textit{on }$\left[ a,b%
\right] $, and $f^{\prime }$ be $\Phi ^{^{\prime }}$-superadditive. Let $%
\Phi :\left[ 0,b-a\right] \rightarrow 
\mathbb{R}
_{+}$ $,$be continuously differentiable, $\Phi ^{^{\prime }}\geq 0$\ and $%
\Phi \left( 0\right) =0$. Then the function $D:\left[ a,b\right] \rightarrow 
\mathbb{R}
$ defined by%
\begin{equation}
D\left( y\right) =f\left( y\right) -f\left( z\right) -f^{\prime }\left(
z\right) \left( y-z\right) -\Phi \left( \left\vert y-z\right\vert \right)
\label{2.1}
\end{equation}%
\textit{is nonnegative on }$\left[ a,b\right] $\textit{, nonincreasing on }$%
\left[ a,z\right) $\textit{, and nondecreasing on }$\left[ z,b\right) $%
\textit{, for }$a\leq z<b$,\textit{\ and }$f$ satisfies (\ref{1.2}) (which
includes uniformly convex functions).\textit{\ }
\end{lemma}

\begin{proof}
As $f^{\prime }$ is $\Phi ^{^{\prime }}$-superadditive, therefore, if $a\leq
z\leq y<b$ then

\begin{equation*}
0\leq \tint\limits_{z}^{y}\left( f^{\prime }(t)-f^{\prime }(z)-\Phi ^{\prime
}(t-z)\right) \mathrm{d}t=f(y)-f(z)-f^{\prime }(z)(y-z)-\Phi (y-z)
\end{equation*}%
and if $0\leq a\leq y\leq z<b$ then%
\begin{equation*}
0\leq \tint\limits_{y}^{z}\left( f^{\prime }(z)-f^{\prime }(t)-\Phi ^{\prime
}(z-t)\right) \mathrm{d}t=f^{\prime }(z)(z-y)-f(z)+f(y)-\Phi (z-y).
\end{equation*}%
Together these show that for any $y,z\in \left[ a,b\right) $,%
\begin{equation*}
f\left( y\right) -f\left( z\right) -f^{\prime }\left( z\right) \left(
y-z\right) -\Phi \left( \left\vert y-z\right\vert \right) \geq 0,
\end{equation*}%
so, we conclude that $D$ is nonnegative on $\left[ a,b\right) $ and
according to Remark \ref{Rem3} $f$\ satisfies (\ref{1.2}) which includes
uniformly convex functions with modulus $\Phi :\left[ 0,b-a\right]
\rightarrow 
\mathbb{R}
_{+}$.\newline

From%
\begin{equation*}
D^{\prime }\left( y\right) =f^{\prime }\left( y\right) -f^{\prime }\left(
z\right) -\Phi ^{\prime }\left( \left\vert y-z\right\vert \right) sgn\left(
y-z\right) ,
\end{equation*}%
as $f^{\prime }$ is $\Phi ^{^{\prime }}$-superadditive for $a\leq y\leq z$
we have%
\begin{equation*}
D^{\prime }\left( y\right) =f^{\prime }\left( y\right) -f^{\prime }\left(
z\right) +\Phi ^{^{\prime }}\left( z-y\right) \leq 0
\end{equation*}%
and similarly for $z\leq y<b$ we have%
\begin{equation*}
D^{\prime }\left( y\right) =f^{\prime }\left( y\right) -f^{\prime }\left(
z\right) -\Phi ^{\prime }\left( y-z\right) \geq 0.
\end{equation*}%
This completes the proof.\bigskip
\end{proof}

Now we present the main results of this section where we show that
inequality (\ref{1.1}) is satisfied, not only for nonnegative coefficients
but also when 
\begin{equation}
0\leq \sum_{i=1}^{j}a_{i}\leq \sum_{i=1}^{n}a_{i},\quad j=1,...,n,\quad
\sum_{i=1}^{n}a_{i}>0  \label{2.2}
\end{equation}%
is satisfied (called Jensen-Steffensen's coefficients) when $x_{1}\leq
x_{2}\leq ...\leq x_{n}$.\bigskip

\begin{theorem}
\label{Th6} Let $f$ be continuously differentiable on $\left[ a,b\right] $,
and $f^{\prime }$ be $\Phi ^{^{\prime }}$-superadditive. Let also $\Phi :%
\left[ 0,b-a\right] \rightarrow 
\mathbb{R}
_{+}$ be continuously differentiable $\Phi ^{^{\prime }}\geq 0$ and $\Phi
\left( 0\right) =0$ Let $\boldsymbol{a}$ be a real $n$-tuple \textit{%
satisfying} (\ref{2.2}) and $x_{i}\in \left[ a,b\right] $, $i=1,...,n$ be
such that $x_{1}\leq x_{2}\leq ...\leq x_{n}$. Then%
\begin{equation}
f\left( c\right) +f^{\prime }\left( c\right) \left( \overline{x}-c\right) +%
\frac{1}{A_{n}}\sum_{i=1}^{n}a_{i}\Phi \left( \left\vert x_{i}-c\right\vert
\right) \leq \frac{1}{A_{n}}\sum_{i=1}^{n}a_{i}f\left( x_{i}\right)
\label{2.3}
\end{equation}%
holds for all $c\in \left[ a,b\right] $ where $A_{i}=\sum_{j=1}^{i}a_{j},$\ $%
\overline{x}=\frac{1}{A_{n}}\sum_{i=1}^{n}a_{i}x_{i}$.

In particular when $c=\frac{1}{A_{n}}\sum_{i=1}^{n}a_{i}x_{i}$, we get the
inequality 
\begin{equation}
\frac{1}{A_{n}}\sum_{i=1}^{n}a_{i}f\left( x_{i}\right) \geq f\left( 
\overline{x}\right) +\frac{1}{A_{n}}\sum_{i=1}^{n}a_{i}\Phi \left(
\left\vert x_{i}-\overline{x}\right\vert \right) .  \label{2.4}
\end{equation}

If in addition $\Phi $ is convex, then 
\begin{eqnarray}
&&\frac{1}{A_{n}}\sum_{i=1}^{n}a_{i}f\left( x_{i}\right) -f\left( c\right)
-f^{\prime }\left( c\right) \left( \overline{x}-c\right)  \label{2.5} \\
&\geq &\frac{1}{A_{n}}\sum_{i=1}^{n}a_{i}\Phi \left( \left\vert
x_{i}-c\right\vert \right) \geq 0.  \notag
\end{eqnarray}
\end{theorem}

\begin{proof}
It was proved in \cite{ASIP} that $x_{1}\leq \overline{x}\leq x_{n}$.

Let $D\left( x_{i}\right) =f\left( x_{i}\right) -f\left( c\right)
-f^{^{\prime }}\left( c\right) \left( x_{i}-c\right) -\Phi \left( \left\vert
x_{i}-c\right\vert \right) ,$ \ $i=1,...,n$.

From Lemma \ref{Lem2} we know that $D(x_{i})\geq 0$ for all $i=1,...,n.$
Comparing $c$ with $x_{1},...,x_{n}$ we must consider three cases.

\textbf{Case 1}. \ \ \ $x_{n}<c<b:$

In this case$\ x_{i}\in \left[ 0,c\right) $ for all $i=1,...,n$, hence,
according to Lemma \ref{Lem2} $\ $\linebreak $D\left( x_{1}\right) \geq
D\left( x_{2}\right) \geq ...\geq D\left( x_{n}\right) \geq 0.$\newline
Denoting $A_{0}=0$ we get that%
\begin{equation*}
a_{i}=A_{i}-A_{i-1},\qquad i=1,...,n
\end{equation*}%
and therefore%
\begin{align*}
\sum_{i=1}^{n}a_{i}D\left( x_{i}\right) & =\sum_{i=1}^{n}\left(
A_{i}-A_{i-1}\right) D\left( x_{i}\right) \\
& =A_{1}D\left( x_{1}\right) +\left( A_{2}-A_{1}\right) D\left( x_{2}\right)
+...+\left( A_{n}-A_{n-1}\right) D\left( x_{n}\right) \\
& =\sum_{i=1}^{n-1}A_{i}\left( D\left( x_{i}\right) -D\left( x_{i+1}\right)
\right) +A_{n}D\left( x_{n}\right) \geq 0.
\end{align*}

\textbf{Case 2. }\ \ \ $a\leq c\leq x_{1}$

In this case $x_{i}\in \left( c,b\right) $ for all $i=1,...,n,$ hence 
\begin{equation*}
0\leq D\left( x_{1}\right) \leq D\left( x_{2}\right) \leq ...\leq D\left(
x_{n}\right) .
\end{equation*}%
\newline
Denoting $\overline{A}_{n+1}=0$ we get that%
\begin{equation*}
\overline{A}_{k}=\sum_{i=k}^{n}a_{i}=A_{n}-A_{k-1},\qquad k=1,...,n,
\end{equation*}%
and%
\begin{equation*}
a_{i}=\overline{A}_{i}-\overline{A}_{i+1},\qquad i=1,...,n.
\end{equation*}%
Therefore%
\begin{align*}
\sum_{i=1}^{n}a_{i}D\left( x_{i}\right) & =\sum_{i=1}^{n}\left( \overline{A}%
_{i}-\overline{A}_{i+1}\right) D\left( x_{i}\right) \\
& =\overline{A}_{1}D\left( x_{1}\right) +\sum_{i=2}^{n}\overline{A}%
_{i}\left( D\left( x_{i}\right) -D\left( x_{i-1}\right) \right) \geq 0.
\end{align*}

\textbf{Case 3}. \ \ $x_{1}\leq c\leq x_{n}$

In this case there exists $k\in \{1,...,n-1\}$ such that $x_{k}\leq c\leq
x_{k+1}.$\newline
By Lemma\ \ref{Lem2} we get that%
\begin{equation*}
D\left( x_{1}\right) \geq D\left( x_{2}\right) \geq ...\geq D\left(
x_{k}\right) \geq 0
\end{equation*}%
and 
\begin{equation*}
0\leq D\left( x_{k+1}\right) \leq D\left( x_{k+2}\right) \leq ...\leq
D\left( x_{n}\right) ,
\end{equation*}%
and%
\begin{align*}
\sum_{i=1}^{n}a_{i}D\left( x_{i}\right) & =\sum_{i=1}^{k}a_{i}D\left(
x_{i}\right) +\sum_{i=k+1}^{n}a_{i}D\left( x_{i}\right) \\
& =\sum_{i=1}^{k-1}A_{i}\left( D\left( x_{i}\right) -D\left( x_{i+1}\right)
\right) +A_{k}D\left( x_{k}\right) \\
& +\overline{A}_{k+1}D\left( x_{k+1}\right) +\sum_{i=k+2}^{n}\overline{A}%
_{i}\left( D\left( x_{i}\right) -D\left( x_{i-1}\right) \right) \\
& \geq 0.
\end{align*}%
From these three cases we get that 
\begin{align*}
\sum_{i=1}^{n}a_{i}D\left( x_{i}\right) & =\sum_{i=1}^{n}a_{i}\left[ f\left(
x_{i}\right) -f\left( c\right) -f^{\prime }\left( c\right) \left(
x_{i}-c\right) -\Phi \left( \left\vert x_{i}-c\right\vert \right) \right] \\
& \geq 0
\end{align*}%
and therefore $($\ref{2.3}) holds. When $c=\overline{x}$ then (\ref{2.4})
holds.

It is given in this paper that $\Phi $ is increasing, $\Phi \left( 0\right)
=0$. If $\Phi $\ is also convex, we only need to show that under our
conditions%
\begin{equation}
\frac{1}{A_{n}}\sum_{i=1}^{n}a_{i}\Phi \left( \left\vert x_{i}-c\right\vert
\right) \geq 0  \label{2.6}
\end{equation}%
holds.

We show that (\ref{2.6}) holds in the case that $x_{k}\leq c\leq x_{k+1}$,\ $%
k=1,...,n-1.$

We use the identity 
\begin{align}
&  \label{2.7} \\
& \sum_{i=1}^{n}a_{i}\Phi \left( \left\vert x_{i}-c\right\vert \right) 
\notag \\
& =\sum_{i=1}^{k-1}A_{i}\left( \Phi \left( c-x_{i}\right) -\Phi \left(
c-x_{i+1}\right) \right) +A_{k}\Phi (c-x_{k})+\overline{A}_{k+1}\Phi
(x_{k+1}-c)  \notag \\
& +\sum_{i=k+2}^{n}\overline{A}_{i}(\Phi \left( x_{i}-c\right) -\Phi \left(
x_{i-1}-c\right) ).  \notag
\end{align}%
As $\Phi $ is nonnegative, convex and $\Phi (0)=0$\ it follows that for $%
a\leq x_{i}\leq x_{i+1}\leq c,$ $i=1,...,k-1,$%
\begin{equation*}
\Phi \left( c-x_{i}\right) -\Phi \left( c-x_{i+1}\right) \geq \Phi
(x_{i+1}-x_{i})-\Phi (0)=\Phi (x_{i+1}-x_{i})\geq 0
\end{equation*}%
and for $c\leq x_{i-1}\leq x_{i}<b,$ \ $i=k+2,...,n$,%
\begin{equation*}
\Phi \left( x_{i}-c\right) -\Phi \left( x_{i-1}-c\right) \geq \Phi
(x_{i}-x_{i-1})-\Phi (0)=\Phi (x_{i}-x_{i-1})\geq 0.
\end{equation*}%
Therefore, (\ref{2.6}) is satisfied and together with (\ref{2.3}) we get
that (\ref{2.5}) holds.\newline
Thus the proof of Theorem \ref{Th6} is complete.\bigskip\ 
\end{proof}

\section{\protect\bigskip \textbf{Jensen type results on uniformly convex, }$%
\Phi $\textbf{-convex and superquadratic functions}}

In Theorem \ref{Th1},( \cite[Theorem 4]{BL}) the authors deal with strongly
convex functions. In this section we show that\ by a technique as there, but
here for functions $f$ that satisfy

\begin{equation}
f\left( x\right) -f\left( u\right) \geq \varphi \left( u\right) \left(
x-u\right) +\Phi \left( \left\vert x-u\right\vert \right) ,  \label{3.1}
\end{equation}%
we get new results for the following:

a. \ \ For superquadratic\ $f$ where we replace in (\ref{3.1}) $\Phi $ with $%
f$\ , (see Definition \ref{Def3}),

b. \ \ For uniformly convex $f$ where in (\ref{3.1}) we replace $\varphi $ \
with $f^{^{\prime }}$, $\Phi \geq 0$ (see Remark \ref{Rem3}$),$

c. \ \ For $\Phi $-convex $f$\ where in (\ref{3.1}) we replace $\Phi $ with $%
\left( -\Phi \right) $ (see Corollary \ref{Cor1}).

First we get the following four theorems about \textbf{uniformly convex}
functions:

\begin{theorem}
\label{Th7} \ Let $f:\left( a,b\right) \rightarrow 
\mathbb{R}
$ be a uniformly convex function with modulus $\Phi \geq 0$. Suppose $%
\mathbf{x}=\left( x_{1},...,x_{n}\right) \in \left( a,b\right) ^{n\text{ }}$
and $\mathbf{a}=\left( a_{1},...,a_{n}\right) $ is a nonnegative $n$-tuple
with $A_{n}=\sum_{i=1}^{n}a_{i}>0$. Let $\overline{x}=\frac{1}{A_{n}}%
\sum_{i=1}^{n}a_{i}x_{i},$ $i\in \left\{ 1,...,n\right\} $. Then%
\begin{eqnarray}
&&  \label{3.2} \\
0 &\leq &\left\vert \frac{1}{A_{i}}\sum_{i=1}^{n}a_{i}\left\vert f\left(
x_{i}\right) -f\left( \overline{x}\right) -\Phi \left( \left\vert \overline{x%
}-x_{i}\right\vert \right) \right\vert -\frac{1}{A_{n}}\sum_{i=1}^{n}a_{i}%
\left\vert f^{^{\prime }}\left( \overline{x}\right) \right\vert \left\vert 
\overline{x}-x_{i}\right\vert \right\vert  \notag \\
&\leq &\frac{1}{A_{n}}\sum_{i=1}^{n}a_{i}f\left( x_{i}\right) -f\left( 
\overline{x}\right) -\frac{1}{A_{n}}\sum_{i=1}^{n}a_{i}\Phi \left(
\left\vert \overline{x}-x_{i}\right\vert \right) .  \notag
\end{eqnarray}
\end{theorem}

\begin{proof}
Applying the triangle inequality $\left\vert \left\vert u\right\vert
-\left\vert v\right\vert \right\vert \leq \left\vert u-v\right\vert $ to%
\begin{equation*}
f\left( x\right) -f\left( y\right) -\Phi \left( \left\vert x-y\right\vert
\right) -f^{^{\prime }}\left( y\right) \left( x-y\right) \geq 0
\end{equation*}%
satisfied by uniforlmly convex functions with modulus $\Phi $, we get 
\begin{eqnarray}
&&\left\vert \left\vert f\left( x\right) -f\left( y\right) -\Phi \left(
\left\vert x-y\right\vert \right) \right\vert -\left\vert f^{^{\prime
}}\left( y\right) \left\vert x-y\right\vert \right\vert \right\vert
\label{3.3} \\
&\leq &\left\vert f\left( x\right) -f\left( y\right) -\Phi \left( \left\vert
x-y\right\vert \right) -f^{^{\prime }}\left( y\right) \left( x-y\right)
\right\vert  \notag \\
&=&f\left( x\right) -f\left( y\right) -\Phi \left( \left\vert x-y\right\vert
\right) -f^{^{\prime }}\left( y\right) \left( x-y\right) .  \notag
\end{eqnarray}%
Setting $y=\overline{x}$ and $x=x_{i}$, $i\in \left\{ 1,...,n\right\} $, we
have%
\begin{eqnarray}
&&\left\vert \left\vert f\left( x_{i}\right) -f\left( \overline{x}\right)
-\Phi \left( \left\vert \overline{x}-x_{i}\right\vert \right) \right\vert
-\left\vert f^{^{\prime }}\left( \overline{x}\right) \right\vert \left\vert 
\overline{x}-x_{i}\right\vert \right\vert  \label{3.4} \\
&\leq &\left\vert f\left( x_{i}\right) -f\left( \overline{x}\right)
-f^{^{\prime }}\left( \overline{x}\right) \left( \overline{x}-x_{i}\right)
-\Phi \left( \left\vert \overline{x}-x_{i}\right\vert \right) \right\vert 
\notag \\
&=&f\left( x_{i}\right) -f\left( \overline{x}\right) -f^{^{\prime }}\left( 
\overline{x}\right) \left( \overline{x}-x_{i}\right) -\Phi \left( \left\vert 
\overline{x}-x_{i}\right\vert \right)  \notag
\end{eqnarray}%
Now multiplying (\ref{3.4}) by $a_{i}$, summing over $i,$ $i=1,...,n$, and
dividing by $A_{n}=\sum_{i=1}^{n}a_{i}>0$, we get 
\begin{eqnarray}
&&\frac{1}{A_{n}}\sum_{i=1}^{n}a_{i}\left\vert \left\vert f\left(
x_{i}\right) -f\left( \overline{x}\right) -\Phi \left( \left\vert \overline{x%
}-x_{i}\right\vert \right) \right\vert -\left\vert f^{^{\prime }}\left( 
\overline{x}\right) \right\vert \left\vert \overline{x}-x_{i}\right\vert
\right\vert  \label{3.5} \\
&\leq &\frac{1}{A_{n}}\sum_{i=1}^{n}a_{i}\left\vert f\left( x_{i}\right)
-f\left( \overline{x}\right) -f^{^{\prime }}\left( \overline{x}\right)
\left( \overline{x}-x_{i}\right) -\Phi \left( \left\vert \overline{x}%
-x_{i}\right\vert \right) \right\vert  \notag \\
&=&\frac{1}{A_{n}}\sum_{i=1}^{n}a_{i}f\left( x_{i}\right) -\frac{1}{A_{n}}%
\sum_{i=1}^{n}a_{i}f\left( \overline{x}\right)  \notag \\
&&-\frac{1}{A_{n}}\sum_{i=1}^{n}a_{i}f^{^{\prime }}\left( \overline{x}%
\right) \left( \overline{x}-x_{i}\right) -\frac{1}{A_{n}}\sum_{i=1}^{n}a_{i}%
\Phi \left( \left\vert \overline{x}-x_{i}\right\vert \right) .  \notag
\end{eqnarray}

Because $\sum_{i=1}^{n}\left( \overline{x}-x_{i}\right) =0$, the
inequalities (\ref{3.5}) and (\ref{3.2}) are the same and the proof of the
theorem is complete.
\end{proof}

With the same technique as in Theorem \ref{Th7}\ we get the other three
theorems about uniformly convex functions:

\begin{theorem}
\label{Th8}\ Let $f:\left( a,b\right) \rightarrow 
\mathbb{R}
$ be a uniformly convex function with modulus $\Phi :\left[ 0,b-a\right]
\rightarrow 
\mathbb{R}
_{+}$. Suppose $\mathbf{x}=\left( x_{1},...,x_{n}\right) \in \left(
a,b\right) ^{n}$ and $\mathbf{a}=\left( a_{1},...,a_{n}\right) $ is a
nonnegative $n$-tuple with $A_{n}=\sum_{i=1}^{n}a_{i}>0$. Let $\overline{x}=%
\frac{1}{A_{n}}\sum_{i=1}^{n}a_{i}x_{i}$ and $\widehat{x}_{i}=\left(
1-\lambda _{i}\right) \overline{x}+\lambda _{i}x_{i},$ $\lambda _{i}\in %
\left[ 0,1\right] ,$ $i\in \left\{ 1,...,n\right\} $. Then%
\begin{eqnarray*}
0 &\leq &\left\vert \frac{1}{A_{i}}\sum_{i=1}^{n}a_{i}\left\vert f\left(
x_{i}\right) -f\left( \widehat{x}_{i}\right) -\Phi \left( \left( 1-\lambda
_{i}\right) \left( \left\vert \overline{x}-x_{i}\right\vert \right) \right)
\right\vert \right. \\
&&\left. -\frac{1}{A_{n}}\sum_{i=1}^{n}a_{i}\left( 1-\lambda _{i}\right)
\left\vert f^{^{\prime }}\left( \widehat{x}_{i}\right) \right\vert
\left\vert \overline{x}-x_{i}\right\vert \right\vert \\
&\leq &\frac{1}{A_{n}}\sum_{i=1}^{n}a_{i}fx_{i}-\frac{1}{A_{n}}%
\sum_{i=1}^{n}a_{i}f\left( \widehat{x}_{i}\right) \\
&&-\frac{1}{A_{n}}\sum_{i=1}^{n}a_{i}\left( 1-\lambda _{i}\right)
f^{^{\prime }}\left( \widehat{x}_{i}\right) \left( \overline{x}-x_{i}\right)
-\sum_{i=1}^{n}a_{i}\Phi \left( \left( 1-\lambda _{i}\right) \left(
\left\vert \overline{x}-x_{i}\right\vert \right) \right) .
\end{eqnarray*}%
hold.
\end{theorem}

Similarly, we get an inequality, which counterparts the Jensen inequality
for uniformly convex functions.

\begin{theorem}
\label{Th9} \ Let Let $f:\left( a,b\right) \rightarrow 
\mathbb{R}
$ be a uiformly convex function with modulus $\Phi $. Suppose $\mathbf{x}%
=\left( x_{1},...,x_{n}\right) \in \left( a,b\right) ^{n},$ $\Phi :\left[
0,b-a\right] \rightarrow 
\mathbb{R}
_{+}$, and $\mathbf{a}=\left( a_{1},...,a_{n}\right) $ is a nonnegative $n$%
-tuple with $A_{n}=\sum_{i=1}^{n}a_{i}>0$ and $\overline{x}=\frac{1}{A_{n}}%
\sum_{i=1}^{n}a_{i}x_{i}$. Then%
\begin{eqnarray}
&&\frac{1}{A_{n}}\sum_{i=1}^{n}a_{i}f\left( x_{i}\right) -f\left( \overline{x%
}\right)  \label{3.6} \\
&\leq &\frac{1}{A_{n}}\sum_{i=1}^{n}a_{i}f^{^{\prime }}\left( x_{i}\right)
\left( x_{i}-\overline{x}\right) -\frac{1}{A_{n}}\sum_{i=1}^{n}a_{i}\Phi
\left( \left\vert \overline{x}-x_{i}\right\vert \right) .  \notag
\end{eqnarray}
\end{theorem}

\begin{proof}
For $x=\overline{x}$ and $y=x_{i},$ $i\in \left\{ 1,...,n\right\} ,$ we have%
\begin{equation*}
f\left( \overline{x}\right) -f\left( x_{i}\right) \geq f^{^{\prime }}\left(
x_{i}\right) \left( \overline{x}-x_{i}\right) +\Phi \left( \left\vert 
\overline{x}-x_{i}\right\vert \right) .
\end{equation*}%
Now multiplying by $a_{i}$, summing over $i,$ $i=1,...,n$, and then dividing
by $A_{n}>0$, we get 
\begin{equation*}
f\left( \overline{x}\right) -\frac{1}{A_{n}}\sum_{i=1}^{n}a_{i}f\left(
x_{i}\right) \geq \frac{1}{A_{n}}\sum_{i=1}^{n}a_{i}f^{^{\prime }}\left(
x_{i}\right) \left( \overline{x}-x_{i}\right) +\frac{1}{A_{n}}%
\sum_{i=1}^{n}a_{i}\Phi \left( \left\vert \overline{x}-x_{i}\right\vert
\right) ,
\end{equation*}%
which is equivalent to (\ref{3.6}).
\end{proof}

Also, we get :

\begin{theorem}
\label{Th10}\ Let $f:\left( a,b\right) \rightarrow 
\mathbb{R}
$ be a uniformly convex function with modulus $\Phi $. Suppose $\mathbf{x}%
=\left( x_{1},...,x_{n}\right) \in \left( a,b\right) ^{n}$ and $\mathbf{a}%
=\left( a_{1},...,a_{n}\right) $ is a nonnegative $n$-tuple with $%
A_{n}=\sum_{i=1}^{n}a_{i}>0$ and $\overline{x}=\frac{1}{A_{n}}%
\sum_{i=1}^{n}a_{i}x_{i}$. Let $\lambda _{i}\in \left[ 0,1\right] ,$ $i\in
\left\{ 1,...,n\right\} $. Then the inequality%
\begin{eqnarray*}
&&\frac{1}{A_{n}}\sum_{i=1}^{n}a_{i}f\left( x_{i}\right) -\frac{1}{A_{n}}%
\sum_{i=1}^{n}a_{i}f\left( \left( 1-\lambda _{i}\right) \overline{x}+\lambda
_{i}x_{i}\right) \\
&\leq &\frac{1}{A_{n}}\sum_{i=1}^{n}a_{i}\left( 1-\lambda _{i}\right)
f^{^{\prime }}\left( x_{i}\right) \left( x_{i}-\overline{x}\right) -\frac{1}{%
A_{n}}\sum_{i=1}^{n}a_{i}\Phi \left( \left( 1-\lambda _{i}\right) \left(
\left\vert \overline{x}-x_{i}\right\vert \right) \right) ,
\end{eqnarray*}%
holds.
\end{theorem}

\begin{example}
\label{Ex1} Theorems \ref{Th7}, \ref{Th8}, \ref{Th9} and \ref{Th10} hold
when $f$ is a strongly convex function where $\Phi \left( \left\vert
x-y\right\vert \right) =c\left( x-y\right) ^{2}$, $c>0$ as proved in \cite[%
Theorem 4]{BL} and quoted in Theorem \ref{Th1}.
\end{example}

Similar to the results obtained in Theorems \ref{Th7}, \ref{Th8}, \ref{Th9}
and \ref{Th10}, we get four results for $\Phi $\textbf{-convex functions}.
The proof uses Corollary \ref{Cor1} and we replace in the inequalities
obtained there $\Phi $ with $\left( -\Phi \right) $ and $f^{^{\prime }}$
with $\varphi $:

\begin{theorem}
\label{Th11} \ Let $f:\left( a,b\right) \rightarrow 
\mathbb{R}
$ be a $\Phi $-convex function, with $\Phi :\left[ 0,b-a\right] \rightarrow 
\mathbb{R}
_{+}$. Suppose $\mathbf{x}=\left( x_{1},...,x_{n}\right) \in \left(
a,b\right) ^{n\text{ }}$ and $\mathbf{a}=\left( a_{1},...,a_{n}\right) $ is
a nonnegative $n$-tuple with $A_{n}=\sum_{i=1}^{n}a_{i}>0$. Let $\overline{x}%
=\frac{1}{A_{n}}\sum_{i=1}^{n}a_{i}x_{i},$ $i\in \left\{ 1,...,n\right\} $.
Then%
\begin{eqnarray}
&&  \label{3.7} \\
0 &\leq &\left\vert \frac{1}{A_{i}}\sum_{i=1}^{n}a_{i}\left\vert f\left(
x_{i}\right) -f\left( \overline{x}\right) +\Phi \left( \left\vert \overline{x%
}-x_{i}\right\vert \right) \right\vert -\frac{1}{A_{n}}\sum_{i=1}^{n}a_{i}%
\left\vert \varphi \left( \overline{x}\right) \right\vert \left\vert 
\overline{x}-x_{i}\right\vert \right\vert  \notag \\
&\leq &\frac{1}{A_{n}}\sum_{i=1}^{n}a_{i}f\left( x_{i}\right) -f\left( 
\overline{x}\right) +\sum_{i=1}^{n}a_{i}\Phi \left( \left\vert \overline{x}%
-x_{i}\right\vert \right) .  \notag
\end{eqnarray}
\end{theorem}

\begin{theorem}
\label{Th12} \ Let $f:\left( a,b\right) \rightarrow 
\mathbb{R}
$ be a $\Phi $-convex function with $\Phi :\left[ 0,b-a\right] \rightarrow 
\mathbb{R}
_{+}$. Suppose $\mathbf{x}=\left( x_{1},...,x_{n}\right) \in \left(
a,b\right) ^{n\text{ }}$ and $\mathbf{a}=\left( a_{1},...,a_{n}\right) $ is
a nonnegative $n$-tuple with $A_{n}=\sum_{i=1}^{n}a_{i}>0$. Let $\overline{x}%
=\frac{1}{A_{n}}\sum_{i=1}^{n}a_{i}x_{i}$ and $\widehat{x}_{i}=\left(
1-\lambda _{i}\right) \overline{x}+\lambda _{i}x_{i},$ $\lambda _{i}\in %
\left[ 0,1\right] ,$ $i\in \left\{ 1,...,n\right\} $. Then%
\begin{eqnarray*}
0 &\leq &\left\vert \frac{1}{A_{i}}\sum_{i=1}^{n}a_{i}\left\vert f\left(
x_{i}\right) -f\left( \widehat{x}_{i}\right) +\Phi \left( \left( 1-\lambda
_{i}\right) \left( \left\vert \overline{x}-x_{i}\right\vert \right) \right)
\right\vert \right. \\
&&\left. -\frac{1}{A_{n}}\sum_{i=1}^{n}a_{i}\left( 1-\lambda _{i}\right)
\left\vert \varphi \left( \widehat{x}_{i}\right) \right\vert \left\vert 
\overline{x}-x_{i}\right\vert \right\vert \\
&\leq &\frac{1}{A_{n}}\sum_{i=1}^{n}a_{i}fx_{i}-\frac{1}{A_{n}}%
\sum_{i=1}^{n}A_{i}f\left( \widehat{x}_{i}\right) \\
&&-\frac{1}{A_{n}}\sum_{i=1}^{n}a_{i}\left( 1-\lambda _{i}\right) \varphi
\left( \widehat{x}_{i}\right) \left( \overline{x}-x_{i}\right) +\Phi
\sum_{i=1}^{n}a_{i}\left( 1-\lambda _{i}\right) \left( \left\vert \overline{x%
}-x_{i}\right\vert \right) .
\end{eqnarray*}%
hold.
\end{theorem}

\begin{theorem}
\label{Th13}\ Let Let $f:\left( a,b\right) \rightarrow 
\mathbb{R}
$ be a $\Phi $-convex function with $\Phi :\left[ 0,b-a\right] \rightarrow 
\mathbb{R}
_{+}$. Suppose $\mathbf{x}=\left( x_{1},...,x_{n}\right) \in \left(
a,b\right) ^{n},$ and $\mathbf{a}=\left( a_{1},...,a_{n}\right) $ is a
nonnegative $n$-tuple with $A_{n}=\sum_{i=1}^{n}a_{i}>0$ and $\overline{x}=%
\frac{1}{A_{n}}\sum_{i=1}^{n}a_{i}x_{i}$. Then%
\begin{eqnarray}
&&\frac{1}{A_{n}}\sum_{i=1}^{n}a_{i}f\left( x_{i}\right) -\frac{1}{A_{n}}%
\sum_{i=1}^{n}a_{i}f\left( \overline{x}\right)  \label{3.8} \\
&\leq &\frac{1}{A_{n}}\sum_{i=1}^{n}a_{i}\varphi \left( x_{i}\right) \left(
x_{i}-\overline{x}\right) -\frac{1}{A_{n}}\sum_{i=1}^{n}a_{i}\Phi \left(
\left\vert \overline{x}-x_{i}\right\vert \right) .  \notag
\end{eqnarray}
\end{theorem}

\begin{theorem}
\label{Th14}\ Let $f:\left( a,b\right) \rightarrow 
\mathbb{R}
$ be a $\Phi $-convex function with modulus $\Phi >0$. Suppose $\mathbf{x}%
=\left( x_{1},...,x_{n}\right) \in \left( a,b\right) ^{n}$ and $\mathbf{a}%
=\left( a_{1},...,a_{n}\right) $ is a nonnegative $n$-tuple with $%
A_{n}=\sum_{i=1}^{n}a_{i}>0$ and $\overline{x}=\frac{1}{A_{n}}%
\sum_{i=1}^{n}a_{i}x_{i}$. Let $\lambda _{i}\in \left[ 0,1\right] ,$ $i\in
\left\{ 1,...,n\right\} $. Then%
\begin{eqnarray*}
&&\frac{1}{A_{n}}\sum_{i=1}^{n}a_{i}f\left( x_{i}\right) -\frac{1}{A_{n}}%
\sum_{i=1}^{n}a_{i}f\left( \left( 1-\lambda _{i}\right) \overline{x}+\lambda
_{i}x_{i}\right) \\
&\leq &\frac{1}{A_{n}}\sum_{i=1}^{n}a_{i}\left( 1-\lambda _{i}\right)
\varphi \left( x_{i}\right) \left( x_{i}-\overline{x}\right) +\frac{1}{A_{n}}%
\sum_{i=1}^{n}a_{i}\Phi \left( \left( 1-\lambda _{i}\right) \left( \overline{%
x}-x_{i}\right) \right) .
\end{eqnarray*}
\end{theorem}

Replacing $\left( -\Phi \right) $ with $f$ in Theorems \ref{Th11} - \ref%
{Th14} we get four inequalities for \textbf{superquadratic functions}.

\begin{theorem}
\label{Th15}\ Let $f$ be a superquadratic function. Suppose $\mathbf{x}%
=\left( x_{1},...,x_{n}\right) \geq \mathbf{0}$ and $\mathbf{a}=\left(
a_{1},...,a_{n}\right) $ is a nonnegative $n$-tuple with $%
A_{n}=\sum_{i=1}^{n}a_{i}>0$. Let $\overline{x}=\frac{1}{A_{n}}%
\sum_{i=1}^{n}a_{i}x_{i},$ $i\in \left\{ 1,...,n\right\} $. Then%
\begin{eqnarray*}
0 &\leq &\left\vert \frac{1}{A_{i}}\sum_{i=1}^{n}a_{i}\left\vert f\left(
x_{i}\right) -f\left( \overline{x}\right) -f\left( \left\vert \overline{x}%
-x_{i}\right\vert \right) \right\vert -\frac{1}{A_{n}}\sum_{i=1}^{n}a_{i}%
\left\vert \varphi \left( \overline{x}\right) \right\vert \left\vert 
\overline{x}-x_{i}\right\vert \right\vert \\
&\leq &\frac{1}{A_{n}}\sum_{i=1}^{n}a_{i}f\left( x_{i}\right) -f\left( 
\overline{x}\right) -\sum_{i=1}^{n}a_{i}f\left( \left\vert \overline{x}%
-x_{i}\right\vert \right) .
\end{eqnarray*}
\end{theorem}

\begin{theorem}
\label{Th16}\ Let $f$ be a superquadratic function. Suppose $\mathbf{x}%
=\left( x_{1},...,x_{n}\right) \geq 0$ and $\mathbf{a}=\left(
a_{1},...,a_{n}\right) $ is a nonnegative $n$-tuple with $%
A_{n}=\sum_{i=1}^{n}a_{i}>0$. Let $\overline{x}=\frac{1}{A_{n}}%
\sum_{i=1}^{n}a_{i}x_{i}$ and $\widehat{x}_{i}=\left( 1-\lambda _{i}\right) 
\overline{x}+\lambda _{i}x_{i},$ $\lambda _{i}\in \left[ 0,1\right] ,$ $i\in
\left\{ 1,...,n\right\} $. Then%
\begin{eqnarray*}
0 &\leq &\left\vert \frac{1}{A_{i}}\sum_{i=1}^{n}a_{i}\left\vert f\left(
x_{i}\right) -f\left( \widehat{x}_{i}\right) -f\left( \left( 1-\lambda
_{i}\right) \left( \left\vert \overline{x}-x_{i}\right\vert \right) \right)
\right\vert \right. \\
&&\left. -\frac{1}{A_{n}}\sum_{i=1}^{n}a_{i}\left( 1-\lambda _{i}\right)
\left\vert \varphi \left( \widehat{x}_{i}\right) \right\vert \left\vert 
\overline{x}-x_{i}\right\vert \right\vert \\
&\leq &\frac{1}{A_{n}}\sum_{i=1}^{n}a_{i}fx_{i}-\frac{1}{A_{n}}%
\sum_{i=1}^{n}a_{i}f\left( \widehat{x}_{i}\right) \\
&&-\frac{1}{A_{n}}\sum_{i=1}^{n}a_{i}\left( 1-\lambda _{i}\right) \varphi
\left( \widehat{x}_{i}\right) \left( \overline{x}-x_{i}\right)
-f\sum_{i=1}^{n}a_{i}\left( 1-\lambda _{i}\right) \left( \left\vert 
\overline{x}-x_{i}\right\vert \right) .
\end{eqnarray*}%
hold.
\end{theorem}

\begin{theorem}
\label{Th17}\ Let $f$ be a superquadratic function. Suppose $\mathbf{x}%
=\left( x_{1},...,x_{n}\right) \geq \mathbf{0},$ and $\mathbf{a}=\left(
a_{1},...,a_{n}\right) $ is a nonnegative $n$-tuple with $%
A_{n}=\sum_{i=1}^{n}a_{i}>0$ and $\overline{x}=\frac{1}{A_{n}}%
\sum_{i=1}^{n}a_{i}x_{i}$. Then%
\begin{eqnarray*}
&&\frac{1}{A_{n}}\sum_{i=1}^{n}a_{i}f\left( x_{i}\right) -f\left( \overline{x%
}\right) \\
&\leq &\frac{1}{A_{n}}\sum_{i=1}^{n}a_{i}\varphi \left( x_{i}\right) \left(
x_{i}-\overline{x}\right) -\frac{1}{A_{n}}\sum_{i=1}^{n}a_{i}f\left(
\left\vert \overline{x}-x_{i}\right\vert \right) .
\end{eqnarray*}
\end{theorem}

\begin{theorem}
\label{Th18}\ Let $f$ be a superquadratic function. Suppose $\mathbf{x}%
=\left( x_{1},...,x_{n}\right) \geq \mathbf{0}$ and $\mathbf{a}=\left(
a_{1},...,a_{n}\right) $ is a nonnegative $n$-tuple with $%
A_{n}=\sum_{i=1}^{n}a_{i}>0$ and $\overline{x}=\frac{1}{A_{n}}%
\sum_{i=1}^{n}a_{i}x_{i}$. Let $\lambda _{i}\in \left[ 0,1\right] ,$ $i\in
\left\{ 1,...,n\right\} $. Then%
\begin{eqnarray*}
&&\frac{1}{A_{n}}\sum_{i=1}^{n}a_{i}f\left( x_{i}\right) -\frac{1}{A_{n}}%
\sum_{i=1}^{n}a_{i}f\left( \left( 1-\lambda _{i}\right) \overline{x}+\lambda
_{i}x_{i}\right) \\
&\leq &\frac{1}{A_{n}}\sum_{i=1}^{n}a_{i}\left( 1-\lambda _{i}\right)
\varphi \left( x_{i}\right) \left( x_{i}-\overline{x}\right) -\frac{1}{A_{n}}%
\sum_{i=1}^{n}a_{i}f\left( \left( 1-\lambda _{i}\right) \left( \overline{x}%
-x_{i}\right) \right) .
\end{eqnarray*}
\end{theorem}

\section{\textbf{Improvement of Jensen's functional for superquadratic and
uniformly convex functions}}

\begin{theorem}
\bigskip \label{Th19} Under the same conditions and definitions on\ $\ 
\mathbf{p},$ $\mathbf{q},$ $\mathbf{x},$ $m^{\ast }$ and $M^{\ast }$ as in
Theorem \ref{Th5}, if $I$ is $\left[ 0,a\right) $ or $\left[ 0,\infty
\right) $ and $f$ is a superquadratic function on $I$, such that $%
f^{^{\prime }}$ is superadditive, and $x_{1}\leq x_{2}\leq ,...,\leq x_{n}$,
then 
\begin{eqnarray}
&&J_{n}\left( f,\mathbf{x},\mathbf{p}\right) -m^{\ast }J_{n}\left( f,\mathbf{%
x},\mathbf{q}\right)  \label{4.1} \\
&\geq &m^{\ast }f\left( \left\vert \sum_{i=1}^{n}\left( q_{i}-p_{i}\right)
x_{i}\right\vert \right) +\sum_{i=1}^{n}\left( p_{i}-m^{\ast }q_{i}\right)
f\left( \left\vert x_{i}-\sum_{j=1}^{n}p_{j}x_{j}\right\vert \right)  \notag
\end{eqnarray}%
and 
\begin{eqnarray}
&&J_{n}\left( f,\mathbf{x},\mathbf{p}\right) -M^{\ast }J_{n}\left( f,\mathbf{%
x},\mathbf{q}\right)  \label{4.2} \\
&\leq &-\sum_{i=1}^{n}\left( M^{\ast }q_{i}-p_{i}\right) f\left( \left\vert
x_{i}-\sum_{j=1}^{n}q_{j}x_{j}\right\vert \right) -f\left( \left\vert
\sum_{i=1}^{n}\left( p_{i}-q_{i}\right) x_{i}\right\vert \right) .  \notag
\end{eqnarray}
\end{theorem}

\begin{proof}
\ It is clear that $\ m^{\ast }<1,$ and $\ M^{\ast }>1$ when $\mathbf{p}\neq 
\mathbf{q}$.

As \ \ $\sum_{i=1}^{n}q_{i}=1$\ \ and $1\geq \sum_{i=1}^{j}q_{i}>0$ \ it is
easy to verify according to Jensen-Steffensen coefficients, that there is an
integer $k,$\ $2\leq k\leq n$ \ such that $x_{\left( k-1\right) }\leq
\sum_{i=1}^{n}q_{i}x_{i}\leq x_{\left( k\right) }$.

We apply inequality (\ref{1.4}) for the increasing $(n+1)$-tuple $\mathbf{y=}%
\left( y_{1},...,y_{n+1}\right) $ \ 
\begin{equation}
y_{i}=\left\{ 
\begin{array}{ll}
x_{i,} & i=1,...,k-1 \\ 
\sum_{j=1}^{n}q_{j}x_{j},\quad & i=k \\ 
x_{i-1}, & i=k+1,...,n+1%
\end{array}%
\right.  \label{4.3}
\end{equation}%
and 
\begin{equation}
a_{i}=\left\{ 
\begin{array}{ll}
p_{i}-m^{\ast }q_{i}, & i=1,...,k-1 \\ 
m^{\ast }, & i=k \\ 
p_{i-1}-m^{\ast }q_{i-1},\quad & i=k+1,...,n+1%
\end{array}%
\right.  \label{4.4}
\end{equation}%
where $m^{\ast }$ \ is defined in (\ref{1.10}).

It is clear that $\mathbf{a}$ \ satisfies the inequalities 
\begin{equation}
0\leq \sum_{i=1}^{j}a_{i}\leq \sum_{i=1}^{n+1}a_{i}=1,j=1,...,n+1.
\label{4.5}
\end{equation}%
Therefore, (\ref{1.4}) holds for the increasing $(n+1)$-tuple $\mathbf{y}$\
and for a superquadratic function $f$ that its derivative is superadditive.

Hence%
\begin{eqnarray*}
&&\sum_{i=1}^{n}\left( p_{i}-mq_{i}\right) f\left( x_{i}\right) +mf\left(
\sum_{i=1}^{n}q_{i}x_{i}\right) -f\left( \sum_{i=1}^{n}p_{i}x_{i}\right) \\
&=&\sum_{i=1}^{n+1}d_{i}f\left( y_{i}\right) -f\left(
\sum_{i=1}^{n+1}d_{i}y_{i}\right)
\end{eqnarray*}%
\begin{eqnarray*}
&\geq &\sum_{i=1}^{n+1}d_{i}f\left( \left\vert
y_{i}-\sum_{j=1}^{n+1}d_{j}y_{j}\right\vert \right) \\
&=&\sum_{i=1}^{n}\left( p_{i}-mq_{i}\right) f\left( \left\vert
x_{i}-\sum_{j=1}^{n}p_{j}x_{j}\right\vert \right) +mf\left( \left\vert
\sum_{i=1}^{n}\left( p_{i}-q_{i}\right) x_{i}\right\vert \right)
\end{eqnarray*}%
This completes the proof of (\ref{4.1}).

The proof of (\ref{4.2}) is similar:

We define an increasing $(n+1)$-tuple $\mathbf{z}$%
\begin{equation}
z_{i}=\left\{ 
\begin{array}{ll}
x_{i}, & i=1,...,s-1 \\ 
\sum_{j=1}^{n}p_{j}x_{j},\quad & i=s \\ 
x_{i-1}, & i=s+1,...,n+1%
\end{array}%
\right.  \label{4.6}
\end{equation}%
and 
\begin{equation}
b_{i}=\left\{ 
\begin{array}{ll}
q_{i}-\frac{p_{i}}{M^{\ast }}, & i=1,...,s-1 \\ 
\frac{1}{M^{\ast }}, & i=s \\ 
q_{i-1}-\frac{p_{i-1}}{M^{\ast }},\quad & i=s+1,...,n+1\;,%
\end{array}%
\right.  \label{4.7}
\end{equation}%
where $s$\ satisfies $x_{\left( s-1\right) }\leq
\sum_{j=1}^{n}p_{j}x_{j}\leq $\ $x_{\left( s\right) }.$\ \ As $b_{i}$\ \
satisfies (\ref{2.2}) for $i=1,...,n+1$,\ and $\sum_{i=1}^{n+1}b_{i}=1,$\
by\ using inequality (\ref{1.4}), we get inequality(\ref{4.2}).\bigskip
\end{proof}

Similarly we get:

\begin{theorem}
\label{Th20} Let $f$ be continuously differentiable on $\left[ a,b\right) $, 
$f^{\prime }$ be $\Phi ^{^{\prime }}$-superadditive. Let also $\Phi :\left[
0,b-a\right] \rightarrow 
\mathbb{R}
_{+}$ be continuously differentiable, $\Phi ^{^{\prime }}\geq 0$ and $\Phi
\left( 0\right) =0$. Then under the same conditions and definitions on\ $\ 
\mathbf{p},$ $\mathbf{q},$ $\mathbf{x},$ $m^{\ast }$ and $M^{\ast }$ as in
Theorem \ref{Th19}, we get the inequalities%
\begin{eqnarray*}
&&J_{n}\left( f,\mathbf{x},\mathbf{p}\right) -m^{\ast }J_{n}\left( f,\mathbf{%
x},\mathbf{q}\right) \\
&\geq &m^{\ast }\Phi \left( \left\vert \sum_{i=1}^{n}\left(
q_{i}-p_{i}\right) x_{i}\right\vert \right) +\sum_{i=1}^{n}\left(
p_{i}-m^{\ast }q_{i}\right) \Phi \left( \left\vert
x_{i}-\sum_{j=1}^{n}p_{j}x_{j}\right\vert \right)
\end{eqnarray*}%
and 
\begin{eqnarray*}
&&J_{n}\left( f,\mathbf{x},\mathbf{p}\right) -M^{\ast }J_{n}\left( f,\mathbf{%
x},\mathbf{q}\right) \\
&\leq &-\sum_{i=1}^{n}\left( M^{\ast }q_{i}-p_{i}\right) \Phi \left(
\left\vert x_{i}-\sum_{j=1}^{n}q_{j}x_{j}\right\vert \right) -\Phi \left(
\left\vert \sum_{i=1}^{n}\left( p_{i}-q_{i}\right) x_{i}\right\vert \right) .
\end{eqnarray*}
\end{theorem}

\bigskip

\bigskip

\end{document}